\documentclass[11pt,reqno]{amsart}
\usepackage{amsmath,amsthm,amscd,amsfonts,amssymb,color}
\usepackage{cite}
\usepackage{graphicx}
\usepackage{xcolor}
\usepackage[mathscr]{eucal}
\usepackage{upgreek}
\usepackage[bookmarksnumbered,colorlinks,plainpages]{hyperref}
\newtheorem{theorem}{Theorem}[section]

\newtheorem*{Acknowledgement}{\textnormal{\textbf{Acknowledgement}}}
\newtheorem*{Question}{\textnormal{\textbf{Question}}}

\newtheorem{corollary}[theorem]{Corollary}

\newtheorem{note}[theorem]{Note}
\theoremstyle{definition}
\newtheorem{prop}[theorem]{Properties}
\newtheorem{definition}[theorem]{Definition}
\newtheorem{example}[theorem]{Example}
\newtheorem{Open Prob}[theorem]{Open Problem}
\newtheorem{remark}[theorem]{Remark}
\theoremstyle{remark}
\numberwithin{equation}{section}
\def\DJ{\leavevmode\setbox0=\hbox{D}\kern0pt\rlap{\kern.04em\raise.188\ht0\hbox{-}}D}
\begin{document}

\title[Kannan type equicontraction mappings]{On the extended version of Krasnosel'ski\u{i}'s fixed point theorem for Kannan type equicontraction mappings}

\author[S.\ Pal, A.\ Bera, L.K.\ Dey]
{Subhadip Pal$^{1}$, Ashis Bera$^{2}$, Lakshmi Kanta Dey$^{3}$}

\address{{$^{1}$}   Subhadip Pal,
                    Department of Mathematics,
                    National Institute of Technology
                    Durgapur, India}
                    \email{palsubhadip2@gmail.com}
\address{{$^{2}$}   Ashis Bera,
                    School of Advanced Sciences,
                    Department of Mathematics,
                    VIT Chennai, India}
                    \email{beraashis.math@gmail.com}
                    
\address{{$^{3}$}   Lakshmi Kanta Dey,
                    Department of Mathematics,
                    National Institute of Technology
                    Durgapur, India}
                    \email{lakshmikdey@yahoo.co.in}


\begin{abstract}
	
A sufficient condition is established for the existence of a solution to the equation $\mathcal{T}(u,\mathcal{C}(u))=u$, by considering a class of Kannan type equicontraction mappings $\mathcal{T}:\mathcal{A}\times \overline{\mathcal{C}(\mathcal{A})}\to \mathbb{Y}$, where $\mathcal{A}$ is a convex, closed and bounded subset of a Banach space $\mathbb{Y}$ and $\mathcal{C}$ is a compact mapping. To fulfil the desired purpose, we engage the Sadovskii's theorem, involving the measure of noncompactness. The relevance of the acquired results has been illustrated by considering a certain class of initial value problems.
\end{abstract}		
\maketitle

\setcounter{page}{1}

%
%

\section{Introduction and Preliminaries}
	   

The objective of the present article is to extend the Krasnosel'ski\u{i} fixed point Theorem to an implicit form by considering a special class of contraction mappings. Recently, D. Wardowski has provided an extension version by using F-contraction, which is introduced by himself \cite{War1, War2} in $2018$.  Contrary to some recent works \cite{Kara, War} that only deals with the contraction mappings which are continuous in the first co-ordinate, the class of contraction mappings, as considered by us, includes some discontinuous mappings also. In our results, the domain of definition of any generic member (of our class), say $\mathcal{T}$, has been considered over a product space, where one of the factor spaces is heavily dependent on the concerned compact mapping (also continuous). The readers are referred to \cite{Pre, Pac}, for a detailed study on contraction mapping over product spaces.

\medskip

In the next Section, we consider two different classes, corresponding to the domain of the Kannan type equicontraction mapping $\mathcal{T}$, and derive some basic properties concerning those individual classes. We provide examples to conclude that the class of Kannan type equicontraction mappings is independent of the class of general equicontraction mappings and equicontractive singularly mappings, which have been discussed in the literature \cite{Kara, War}. Finally, in the last Section, we discuss the utility of our result in the setting of a Banach space to ensure the existence of a solution to a particular class of initial value problems. Given any Banach space $\mathbb{Y}$, the closed unit ball of $\mathbb{Y}$ is denoted by $\mathcal{B}(0,1):=\{u\in \mathbb{Y}:~\|u\|\leq 1\}$. In our literature, $e_1$ refers the unit vector of the said Banach space and $\pi_1$ refers the projection to the first component. $l_p$ denotes the space of all $p$-power summable sequences of real numbers, with the $p$-norm.

\medskip
	    
At this point, we state the Krasnosel'ski\u{i} fixed point theorem \cite{Kra, Bur, Prz, Vet}, which is the integral theme of the present work: 
\begin{theorem}\cite{Kra}
``Let $\mathcal{H}$ be a nonempty closed, convex, bounded subset of a Banach space $\mathbb{Y}$. Suppose $\mathcal{T},\mathcal{C}:\mathcal{H}\to \mathbb{Y}$ are two mappings such that $\mathcal{T}$ is a contraction and $\mathcal{C}$ is a compact mapping. If for $u,v\in \mathcal{H}, \mathcal{T}(u)+\mathcal{C}(v)\in \mathcal{H},$
then there is a $v\in \mathcal{H}$ satisfies $\mathcal{T}(v) + \mathcal{C}(v) = v."$  
\end{theorem}

Among various existing notions of measure of noncompactness, we use the setting of Hausdorff measure of noncompactness in our result. For more informations about properties of measure of noncompactness and in this direction, the readers are referred to \cite{BG, Kara, BJMSV,War, AKPRS} and the references therein.

\begin{definition}\cite{BJMSV}
``Let $(\zeta,\rho)$ be a metric space and $\Omega$ be a nonempty bounded subset of $\zeta.$ The Hausdorff measure of noncompactness, Kuratowski measure of noncompactness are denoted by $\beta(\Omega)$, $\alpha(\Omega)$ respectively and defined as
\[\beta(\Omega)=\inf~\left\{\varepsilon>0:~\Omega\subset\bigcup\limits_{i=1}^pB(u_i, r_i),~u_i\in \zeta,~r_i<\varepsilon~(i=1, 2, \dots, p),~p\in\mathbb{N}\right\}\] 
and
\[\alpha(\Omega)=\inf~\left\{\varepsilon>0:~\Omega\subset\bigcup\limits_{j=1}^qB_j,~B_j\subset\zeta,~\mathrm{diam}(B_j)<\varepsilon~(j=1, 2, \dots, q),~ q\in\mathbb{N}\right\}\]
respectively."
\end{definition}
We now state some basic properties of $\beta$ that we require in our subsequent section.
\begin{prop}\cite{AKPRS}
\begin{itemize}
\item[a)]$\beta(\Omega)=0 \Longleftrightarrow \Omega$ is totally bounded.
\item[b)]$\Omega_1\subseteq\Omega_2$ implies $\beta(\Omega_1)\leq\beta(\Omega_2)$.
\item[c)]$\beta(\Omega)=\beta(\overline{\Omega})$.
\item[d)]$\beta(\Omega_1\cup\Omega_2)=\max\{\beta(\Omega_1),\beta(\Omega_2)\}$.
\item[e)]$\beta(p\Omega)=|p|\beta(\Omega),~ p\in \mathbb{R}$.
\item[f)]$\beta(\Omega_1+\Omega_2)\leq \beta(\Omega_1)+\beta(\Omega_2)$.
\end{itemize}
\end{prop}

\medskip

Darbo\cite{Dar} originally proved a remarkable result which assimilates the measure of noncompactness with the fixed point theory. Later Sadovskii provided a sufficient condition for a continuous mapping to have a fixed point. Here, we quote the result due to Sadovskii\cite{Sad} which will be used in our first result.   
\begin{theorem}\cite{Sad}
``Let $\mathbb{Y}$ be a Banach space and $\mathcal{H}$ be a nonempty closed convex bounded subset of $\mathbb{Y}$. Suppose $\mathcal{T}:\mathcal{H} \to \mathcal{H}$ is a continuous mapping. If for any nonempty subset $\mathcal{A}$ of $\mathcal{H}$ with $\mu (\mathcal{A}) > 0$ we have
\[\mu (\mathcal{T}(\mathcal{A})) < \mu (\mathcal{A})\] where $\mu$ is a regular measure of noncompactness in $\mathbb{Y},$ then $\mathcal{T}$ has at least one fixed point in $\mathcal{H}."$
\end{theorem}
			  
\section{Main Results}

Let us begin with the definition of Kannan type equicontraction mapping:

\begin{definition}
Let $\mathbb{Y}$ be a Banach space and $\mathbb{S}$ be a closed subset of a complete metric space $(\mathbb{X}, \rho)$. A mapping $\mathcal{T}:\mathbb{Y}\times \mathbb{S}\to \mathbb{Y}$ is called Kannan type equicontraction mapping if for any $u_1,u_2\in \mathbb{Y}$ and $v\in\mathbb{S}$, 
\[\|\mathcal{T}(u_1,v)-\mathcal{T}(u_2,v)\|\leq k(\|\mathcal{T}(u_1,v)-u_1\|+\|\mathcal{T}(u_2,v)-u_2\|)\] 
for some $k\in [0,\frac{1}{2}).$ 
\end{definition} 

\medskip

It is clear that for a Kannan type equicontraction mapping, $\mathcal{T}:\mathcal{A}\times\mathcal{A}\to\mathcal{A}$, the corresponding family $\mathcal{S}_u=\big\{u\mapsto \mathcal{T}(u, v):v\in\mathcal{A}\big\}$ contains the Kannan contractions. 

\medskip

Kannan type equicontraction condition is independent from equicontractive singularly condition and general equicontraction condition. The situation is illustrated in the following example.

\begin{example}\label{2.6}
Let $\mathcal{A}=[0, 1]\subseteq \mathbb{R}$ and let $\mathcal{C}:\mathcal{A}\to\mathcal{A}$ be a compact mapping such that $\mathcal{C}(u)=u.$ Let $\mathcal{T}:\mathcal{A}\times \overline{\mathcal{C}(\mathcal{A})}\to \mathcal{A}$ be a mapping so that 
\[\mathcal{T}(u, v)=\begin{cases}
\frac{u+v}{4}, & \quad (u, v)\in ([0, 1]\times [0, 1])\setminus \{(1, 1)\}, \\
\frac{1}{5}, & \quad(u, v)=(1, 1).
\end{cases}\]
It is easy to see that $\mathcal{T}$ satisfies Kannan type equicontraction condition.

\noindent Indeed, let $u_1,u_2,\in \mathcal{A}$ and let $v\in \overline{\mathcal{C}(\mathcal{A})}$ be fixed. On simplification, we get
\[\big|\mathcal{T}(u_1,v)-u_1\big|+\big|\mathcal{T}(u_2,v)-u_2\big| \geq \left|\frac{3u_1-3u_2}{4}\right|,\]
and  
\[\big|\mathcal{T}(u_1,v)-\mathcal{T}(u_2,v)\big|=\left|\frac{u_1-u_2}{4}\right|.\] 
Consequently,
\[\big|\mathcal{T}(u_1,v)-\mathcal{T}(u_2,v)\big|\leq \frac{2}{5}\Big(\big|\mathcal{T}(u_1,v)-u_1\big|+\big|\mathcal{T}(u_2,v)-u_2\big|\Big).\]
Moreover,
\[\sup_{(u_1,v)\in ([0,1]\times [0,1])\setminus (1,1)}\left|\mathcal{T}(u_1,v)-\frac{1}{5}\right|= \sup_{(u_1,v)\in ([0,1]\times [0,1])\setminus (1,1)}\left|\frac{5u_1+5v-4}{20}\right|=0.3,\]
and, 
\begin{align*}
\inf_{(u_1,v)\in ([0,1]\times [0,1])\setminus (1,1)}\left(\big|\mathcal{T}(u_1,v)-u_1\big|+\frac{4}{5}\right)& =\inf_{(u_1,v)\in ([0,1]\times [0,1])\setminus (1,1)}\left(\left|\frac{3u_1-v}{4}\right|+\frac{4}{5}\right)\\
& =0.8.
\end{align*} 
If we choose $k=\frac{2}{5}$, then for any $u_1,u_2\in \mathcal{A},$
\[\big|\mathcal{T}(u_1,v)-\mathcal{T}(u_2,v)\big|\leq k\Big(\big|\mathcal{T}(u_1,v)-u_1\big|+\big|\mathcal{T}(u_2,v)-u_2\big|\Big),\]
as desired. Since $\mathcal{T}$ is not continuous at $(1, 1)$, $\mathcal{T}$ does not satisfy equicontractive singularly condition as well as general equicontraction condition.
\end{example}

In the next theorem we have used the classical Schauder's theorem to check whether a solution to the equation $\mathcal{T}(u, \mathcal{C}(u))=u$ exists by considering the compact mapping in the sense of Krasnosel'ski\u{i}.

\begin{theorem}\label{th2.4}
Let $\mathcal{A}$ be a bounded, convex, closed subset of a Banach space $\mathbb{Y}$ and $\mathcal{C}:\mathcal{A}\to \upzeta$ be a compact mapping, where $(\upzeta, \rho)$ is a complete metric space. If a mapping $\mathcal{T}:\mathcal{A}\times \overline{\mathcal{C}(\mathcal{A})}\to \mathcal{A}$ satisfies Kannan type equicontraction condition then $\mathcal{T}(u, \mathcal{C}(u))=u$ has a solution in $\mathcal{A}.$ 
\end{theorem}
\begin{proof}
Corresponding to each $v\in \overline{\mathcal{C}(\mathcal{A})}, \mathcal{T}(v):\mathcal{A}\to \mathcal{A}$ is a Kannan contraction and as $\mathcal{A}$ is complete then each $\mathcal{T}(v)$ has a unique fixed point in $\mathcal{A}.$
\[\implies \mathrm{Corresponding~ to~ each~ v\in\overline{\mathcal{C}(\mathcal{A})},~ \mathcal{T}(u, v)=u~ has~ a~ unique~ solution~ in~ \mathcal{A}.}\]    
Let $u=f(v)$ be the fixed point of $\mathcal{T}(v)$ in $\mathcal{A}.$ Define $f:\overline{\mathcal{C}(\mathcal{A})}\to \mathcal{A}$ such that $f(v)=\mathcal{T}(f(v), v)=\mathcal{T}(v)(f(v))$. Now, $f$ is continuous at any point of $\overline{\mathcal{C(\mathcal{A})}}$ is same as saying that the function $\mathcal{T}$ is continuous at the points $(f(v), v)$. The continuity of any mapping $\mathcal{T}(v)$ at its corresponding fixed point $f(v)$ ensures us $\mathcal{T}$ is continuous at the corresponding point $(f(v), v)$. 
Let $f(v_n)$ be a convergent sequence in $\mathcal{A}$ such that $f(v_n)\to f(v)$ as $n\to \infty.$ Now,
\begin{align*}
\big\|\mathcal{T}(v)(f(v_n))-\mathcal{T}(v)(f(v))\big\| & \leq k\Big(\big\|\mathcal{T}(v)(f(v_n))-f(v_n)\big\|+\big\|\mathcal{T}(v)(f(v))-f(v)\big\|\Big)\\
& \leq k\Big(\big\|\mathcal{T}(v)(f(v_n))-\mathcal{T}(v)(f(v))\big\|+\big\|f(v_n)-f(v)\big\|\Big).
\end{align*}
\[\implies \big\|\mathcal{T}(v)(f(v_n))-\mathcal{T}(v)(f(v))\big\|\leq\dfrac{k}{1-k}\big\|f(v_n)-f(v)\big\|\quad\to 0\quad\mathrm{as}\quad n\to\infty.\]
Therefore, the mapping $\mathcal{T}(v)$ is continuous at its fixed point $f(v)$. As $v$ is chosen arbitrarily, any mapping $\mathcal{T}(v)$ is continuous at its fixed point. Consequently, $\mathcal{T}$ is continuous at the points $(f(v), v)$. As a result, $f:\overline{\mathcal{C}(\mathcal{A})}\to \mathcal{A}$ is continuous. Accordingly $f\circ \mathcal{C}$ is also continuous on a closed, convex and bounded set $\mathcal{A}.$ Now, let $\mathcal{S}=\overline{Co(f(\mathcal{C}(\mathcal{A})))}\subseteq \mathcal{A}.$ Applying Schauder fixed point theorem to $f\circ \mathcal{C}$ on the compact set $\mathcal{S}$ we get, there exists $u_0\in\mathcal{A}$ such that $f(\mathcal{C}(u_0))=u_0.$ Therefore, \[\mathcal{T}(u_0, \mathcal{C}(u_0))=\mathcal{T}(f(\mathcal{C}(u_0)), \mathcal{C}(u_0))=f(\mathcal{C}(u_0))=u_0\]
\end{proof}

The following corollary is an immediate application of the above theorem:
\begin{corollary}
Let $\mathcal{A}$ be a bounded, convex, closed subset of a Banach space $\mathbb{Y}$ and $\mathcal{C}:\mathcal{A}\to \upzeta$ be a compact mapping, where $(\upzeta, \rho)$ is a complete metric space. If a mapping $\mathcal{T}:\mathcal{A}\times \overline{\mathcal{C}(\mathcal{A})}\to \mathcal{A}$ is continuous and satisfies Kannan type equicontraction condition then $\mathcal{T}(u, \mathcal{C}(u))=u$ has a solution in $\mathcal{A}.$
\end{corollary}

\begin{example}
Consider $\mathbb{U}:=\mathrm{Span}\{e_1\},$ which is an one dimensional subspace of the Banach space, $l_p$ where $1\leq p\leq \infty$. Define 
\[B_{\mathbb{U}}:=\{te_1:t\in [0,1]\}.\]
Note that $B_{\mathbb{U}}$ is a bounded, convex, closed subset of $l_p.$
Let $\mathcal{C}:B_{\mathbb{U}}\to l_p$ be the inclusion mapping.
Let us define a mapping $\mathcal{T}:B_{\mathbb{U}}\times\mathcal{C}(B_{\mathbb{U}})\to l_p$ such that 
\[\mathcal{T}(x, y)=\begin{cases} \dfrac{1}{16}\pi_1(x+y)e_1, & (x, y)\in (B_{\mathbb{U}}\times B_{\mathbb{U}})\setminus \{(e_1, e_1)\}\\
\dfrac{1}{18}e_1, & (x, y)=(e_1, e_1)
\end{cases}\]
For $x_n=y_n=\left(1+\frac{1}{n}\right)e_1$ and tending $n$ to infinity we get, 
\[(x_n, y_n)\to (e_1, e_1)~ \mathrm{whereas}~ \mathcal{T}(x_n, y_n)=\frac{1}{16}\pi_1(x_n+y_n)e_1\to \frac{1}{8}e_1\ne \mathcal{T}(e_1, e_1).\]
So, $\mathcal{T}$ is not continuous at $(e_1, e_1).$ Let us take two elements $u=t_ue_1$ and $y=t_ye_1$ from $B_{\mathbb{U}}\setminus \{e_1\}$. Now,
\begin{align*}
\big\|\mathcal{T}(u, y)-u\big\| & = \left\|\dfrac{1}{16}\pi_1(u+y)e_1-u\right\|\\
& =\left\|\dfrac{1}{16}\pi_1(u+y)e_1-ue_1\right\|\\
& =\left|\dfrac{t_u+t_y}{16}-t_u\right|\\
& =\left|\dfrac{t_y-15t_u}{16}\right|. 
\end{align*}
Similarly for $v=t_ve_1\in B_{\mathbb{U}}\setminus \{e_1\}$, 
\[\big\|\mathcal{T}(v, y)-v\big\|=\left|\dfrac{t_y-15t_v}{16}\right|.\]
Adding we get, 
\[\big\|\mathcal{T}(u, y)-u\big\|+\big\|\mathcal{T}(v, y)-v\big\|\geq\left|\dfrac{15t_u-15t_v}{16}\right|.\]
And $\big\|\mathcal{T}(u, y)-\mathcal{T}(v, y)\big\|=\left|\dfrac{t_u-t_v}{16}\right|.$ Which implies that 
\[\big\|\mathcal{T}(u, y)-\mathcal{T}(v, y)\big\|\leq\frac{1}{14}\Big(\big\|\mathcal{T}(u, y)-u\big\|+\big\|\mathcal{T}(v, y)-v\big\|\Big).\]
Now, if $(u, y)\in (B_{\mathbb{U}}\times B_{\mathbb{U}})\setminus \{(e_1, e_1)\}$ and $(v, y)=(e_1, e_1)$ we then have
\[\big\|\mathcal{T}(u, y)-\mathcal{T}(v, y)\big\|=\left\|\frac{1}{16}\pi_1(u+y)e_1-\frac{1}{18}e_1\right\|=\left|\frac{t_u+t_y}{16}-\frac{1}{18}\right|~\mathrm{and}\]
$\big\|\mathcal{T}(u, y)-u\big\|=\left\|\dfrac{1}{16}\pi_1(u+y)e_1-u\right\|=\left|\dfrac{t_y-15t_u}{16}\right|, \big\|\mathcal{T}(v, y)-v\big\|=\dfrac{17}{18}.$
It is clear that $t_u, t_v\in [-1, 1]$ and then 
\begin{align*}
\big\|\mathcal{T}(u, y)-\mathcal{T}(v, y)\big\|\leq\sup\big\|\mathcal{T}(u, y)-\mathcal{T}(v, y)\big\| & =\sup\left|\frac{t_u+t_y}{16}-\frac{1}{18}\right|\\ 
& \leq\frac{1}{7}\inf\left(\left|\dfrac{t_y-15t_u}{16}\right|+\frac{17}{18}\right)\\
& =\frac{1}{7}\inf\Big(\big\|\mathcal{T}(u, y)-u\big\|+\big\|\mathcal{T}(v, y)-v\big\|\Big)\\
& \leq\frac{1}{7}\Big(\big\|\mathcal{T}(u, y)-u\big\|+\big\|\mathcal{T}(v, y)-v\big\|\Big).
\end{align*}
So, for any $u, v\in B_{\mathbb{U}},$
\[\big\|\mathcal{T}(u, y)-\mathcal{T}(v, y)\big\|\leq\frac{1}{7}\Big(\big\|\mathcal{T}(u, y)-u\big\|+\big\|\mathcal{T}(v, y)-v\big\|\Big).\]
As a result, we can conclude from the Theorem \ref{th2.4} that $\mathcal{T}(u, \mathcal{C}(u))=u$ has a solution in $B_{\mathbb{U}}$.
\end{example}

We next introduce the definition of $m$-$th~ invariant$ mapping which is important in due course of our events.

\begin{definition}
Let $(\upzeta, \rho)$ be a metric space and $\mathcal{A}(\ne\emptyset)\subseteq\upzeta$. A mapping $\mathcal{T}:\mathcal{A} \to \upzeta$ is said to be $m$-$th~invariant$ mapping if for any fix $q\in [0, \frac{1}{2}), \mathcal{T}(B_i)\subseteq q^mB_i$, for any bounded set $B_i\subseteq \mathcal{A}$.
\end{definition}

\medskip

We now present an extended version of Krasnosel'ski\u{i} fixed point Theorem for Kannan type equicontraction of $1st~invariant$ along with a compact mapping by utilizing Sadovskii's theorem in measure of noncompactness.
		
\begin{theorem}\label{th2.8}
Let $\mathcal{A}$ be a bounded, convex, closed subset of an infinite dimensional Banach space $\mathbb{Y}$ and $\mathcal{C}:\mathcal{A}\to\upzeta$ be a compact mapping, where $(\upzeta, \rho)$ is a complete metric space. If a Kannan type equicontraction mapping  $\mathcal{T}:\mathcal{A}\times\mathcal{C}(\mathcal{A})\to \mathbb{Y}$ satisfies the following conditions:  
\begin{itemize}
\item[a)] The mapping $\eta:\mathcal{A}\to \mathcal{A}$ such that $\eta(u)=\mathcal{T}(u, \mathcal{C}(u))$ is continuous and $1st~invariant$ mapping,
\item[b)] The family, say $\mathcal{S}_v=\big\{v\mapsto \mathcal{T}(u, v):u\in\mathcal{A}\big\}$ is equicontinuous uniformly,
\end{itemize}
then the equation $\mathcal{T}(u, \mathcal{C}(u))=u$ has a solution in $\mathcal{A}.$ 
\end{theorem}
		
\begin{proof}
Let us assume that $G$ be a bounded subset of $\mathcal{A}$ with $\beta (G)>0.$  As $\mathcal{C}(G)$ is relatively compact, we can find some $v_1, v_2, \dots, v_p$ such that
\begin{equation}
\mathcal{C}(G)\subseteq\bigcup\limits_{j=1}^pB(v_j,\delta). \label{1}	
\end{equation}
We can consider a finite collection $\{B(u_i, \mathcal{R})\}, i=1, 2, \dots, n$ such that \[G\subseteq\bigcup\limits_{j=1}^nB(u_i, \mathcal{R})\] 
with $\mathcal{R}>0$ and $\mathcal{R}\leq \beta(G)+\varepsilon$ follows from the definition of Hausdorff measure of noncompactness.	   
Now, given that $\eta$ is $m$-$th~invariant$ mapping, so for some $q\in [0, \frac{1}{2})$ we have, 
\[\eta(B(u_i, \mathcal{R}))\subseteq qB(u_i, \mathcal{R})\] 
i.e, 
\[\mathcal{T}\big(B(u_i, \mathcal{R})\times \mathcal{C}(B(u_i, \mathcal{R}))\big)\subseteq qB(u_i, \mathcal{R}),\]
for any $B(u_i, \mathcal{R})\subseteq \mathcal{A}.$ As the family, $\mathcal{S}_v$ is equicontinuous uniformly then corresponding to arbitrarily chosen $\varepsilon>0$ there exists $\delta>0$ such that for all $u\in\mathcal{A}$ and $v_1, v_2\in \mathcal{C}(\mathcal{A}),$
\begin{equation}
\big\|\mathcal{T}(u, v_1)-\mathcal{T}(u, v_2)\big\|\leq\varepsilon, \quad \mathrm{whenever} \quad d(v_1, v_2)<\delta\label{2}
\end{equation}
holds for any member of $\mathcal{S}_v$.\\
Apparently, 
\begin{equation}
G\subseteq\bigcup\limits_{j=1}^p\bigcup\limits_{i=1}^n\left(B(u_i, \mathcal{R}) \cap \mathcal{C}^{-1}(B(v_j, \delta))\right). \label{eq2.3}
\end{equation}
Now, we are going to prove $\beta (\eta(G))<\beta (G).$ Let us assume that $h\in \eta(G)$ so, $h=\eta(g)=\mathcal{T}(g, \mathcal{C}(g))$ for some $g\in G$. By (\ref{eq2.3}) we can assume two natural numbers $k~(1\leq k\leq n)$ and $l~(1\leq l\leq p)$ such that 
\[g\in B(u_k, \mathcal{R}) ~ \cap ~  \mathcal{C}^{-1}(B(v_l, \delta)).\]
Here, $g\in B(u_k, \mathcal{R})$ implies that $\|g-u_k\|<\mathcal{R}$. Also, $d(\mathcal{C}(g), v_l)<\delta.$ Then for some $\lambda\in [0, \frac{1}{2}),$
\begin{align*}
\big\|h-\mathcal{T}(u_k, v_l)\big\| & =\big\|\mathcal{T}(g, \mathcal{C}(g))-\mathcal{T}(u_k, v_l)\big\|\\
& \leq \big\|\mathcal{T}(g, \mathcal{C}(g))-\mathcal{T}(u_k, \mathcal{C}(g))\big\|+\big\|\mathcal{T}(u_k, \mathcal{C}(g))-\mathcal{T}(u_k, v_l)\big\|\\
& \leq \lambda\Big(\big\|\mathcal{T}(g, \mathcal{C}(g))-g\big\|+\big\|\mathcal{T}(u_k, \mathcal{C}(g))-u_k\big\|\Big)+\varepsilon. 
\end{align*}
Now from the relationship $\mathcal{T}\big(B(u_k, \mathcal{R})\times \mathcal{C}(B(u_k, \mathcal{R}))\big)\subseteq qB(u_k, \mathcal{R})$ we get 
\[\big\|h-\mathcal{T}(u_k, v_l)\big\|< 2q\lambda\mathcal{R}+\varepsilon\leq 2q\lambda(\beta(G)+\varepsilon)+\varepsilon=2q\lambda\beta(G)+(2q\lambda+1)\varepsilon\]
Since $\varepsilon$ is chosen arbitrarily, we then have 
\[\big\|h-\mathcal{T}(u_k, v_l)\big\|<2q\lambda\beta(G)<\beta(G).\]
Letting $\mathcal{R}_1=2q\lambda\beta(G), ~ h\in B(\mathcal{T}(u_k, v_l), \mathcal{R}_1).$ Therefore,
\[\eta(G)\subseteq\bigcup\limits_{i=1}^n\left(\bigcup\limits_{j=1}^pB(\mathcal{T}\left(u_i, v_j\right), \mathcal{R}_1)\right).\]  
Now, the definition of Hausdorff measure of noncompactness leads us to $\beta(\eta(G))<\beta(G)$.
\end{proof}

\begin{note}
It is clear that for finite dimensional Banach space we can conclude the above theorem using Darbo's theorem in measure of noncompactness by ignoring the $1st~invariant$ condition of the mapping $\eta$ along with the last condition, $(b)$.   
\end{note}

\begin{proof}
For a finite dimensional Banach space $\mathbb{Y}$ the bounded, convex, closed subset, $\mathcal{A}$ of $\mathbb{Y}$ is compact and hence totally bounded. So, any bounded subset, $G$ of $\mathcal{A}$ is also totally bounded. Hence the Kuratowski measure of noncompactness of $G$, $\alpha(G)=0.$ Now,
\begin{align*}
\alpha(G)=0 & \implies\overline{G} ~ is ~ compact\\
& \implies\eta(\overline{G}) ~ is ~ compact~ (as,~\eta~is~continuous~in~the~above~theorem)\\
& \implies\alpha(\eta(\overline{G}))=0\\
& \implies\alpha(\eta(G))=0~(as,~\eta(G)\subseteq\eta(\overline{G}))
\end{align*}
Hence there always exists $k_1>0$ such that $\alpha(\eta(G))\leq k_1\alpha(G).$ Consequently, Darbo's theorem conclude that the equation $\mathcal{T}(u, \mathcal{C}(u))=u$ has a solution in $\mathcal{A}.$
\end{proof}

\begin{note}
To apply Sadovskii's theorem in measure of noncompactness it is necessary to consider the continuity condition of the mapping $\eta$ in Theorem \ref{th2.8}.  
\end{note}
	
\begin{remark}
\begin{itemize}
\item[(i)] The continuity of the mapping $\eta$, as mentioned above does not ensure the continuity of the mapping $\mathcal{T}.$ Indeed, if we consider the mapping $\mathcal{T}$ as in Example \ref{2.6} and the compact mapping as $\mathcal{C}(u)=\frac{u}{4},$ then the mapping $\eta$ is continuous. However, for the inclusion map $\mathcal{C}$, $\mathcal{T}$ as well as $\eta$ fails to be continuous in its domain of definition. In particular, we are dealing with the continuity of $\mathcal{T}$ only on the set $\{(u, \mathcal{C}(u)):u\in\mathcal{A}\}.$
\item[(ii)] The above Theorem \ref{th2.8} can also be concluded by using Darbo's theorem.
\end{itemize}
\end{remark}

\section{Application to Initial Value Problems}

In this section, we furnish a class of initial value problems that can be efficiently solved by applying the results obtained in the preceding section. 

\subsection{Application I} Let us consider an initial value problem
\[\frac{d^2u}{dt^2}+\omega^2u=f(t, u(t))\int_{0}^tG(t, s)u(s)ds,\]
\[u(0)=a, u'(0)=b\]
where $\omega\neq 0, t\in I=[0,1],$ also,
\[G(t, s)=\dfrac{1}{\omega}\sin(\omega(t-s))H(t-s), ~\omega\neq 0\]
and $H(t-s)$ is heaviside step function. Let $u(t)\in \mathbb{Y}$, the vector space of all real valued continuous functions over $I$ equipped with the norm:
\[\|u(t)\|_{\mathbb{Y}}=\sup_{t\in I}\left|u(t)\right|e^{-\gamma t},\quad \gamma\in\mathbb{R}.\]
Note that $(\mathbb{Y}, \|.\|_{\mathbb{Y}})$ forms a Banach space.

\medskip

The following theorem illustrates that the mentioned initial value problem has a solution in the closed unit ball $\mathcal{B}(0,1)$.
	
\begin{theorem}
Let $f:I\times\mathbb{R}\to \mathbb{R}$ be a nonzero function. There exists $M>0$ such that $\big|f(t,u)\big|\leq M,$ for all $t\in I$ and $u\in \mathbb{R}.$ The above initial value problem has a solution in closed unit ball $\mathcal{B}(0, 1)$ if the following conditions are satisfied
\begin{itemize}
\item[(i)]
$\big|f(s, u(s))-f(s, v(s))\big|\leq D(t, u(s), v(s))$ ~ where, 
\begin{align*}
D(t, u(s), v(s))= & \left|u_1(s)-a\cos(\omega t)-b\sin(\omega t)-\int_{0}^tG(t, s)f(s, u_1(s))v(s)ds\right|+\\ & \left|u_2(s)-a\cos(\omega t)-b\sin(\omega t)-\int_{0}^tG(t, s)f(s, u_2(s))v(s)ds\right|.
\end{align*}
\item[(ii)] $M\leq |\omega|^2\left(1-|a|-|b|\right)$.
\item[(iii)] $|\gamma +2\omega|<\frac{|\omega(\gamma^2+\omega^2)|}{2}$.
\end{itemize}	  
\end{theorem}
	
\begin{proof}	
Consider an operator $\mathcal{C}:\mathbb{Y}\to \mathbb{Y}$ such that 
\[\mathcal{C}(u(t))=\int_{0}^tG(t, s)u(s)ds, \quad u\in \mathbb{Y}\]
where $G(t, s)=\frac{1}{\omega}\sin(\omega(t-s))H(t-s)=\begin{cases}
     0, & 0\leq t<s \\
     \frac{1}{\omega}\sin(\omega(t-s)), & s<t<\infty
\end{cases}, ~\omega\neq 0.$\\
Now, for each $u(t)\in \mathcal{B}(0, 1),$ \[\big|\mathcal{C}(u(t))\big|=\left|\int_{0}^tG(t, s)u(s)ds\right|\leq \frac{1}{|\omega|}.\]

Then $\mathcal{C}$ maps $\mathcal{B}(0, 1)$ to $\mathcal{B}\left(0, \dfrac{1}{|\omega|}\right).$ Thus, we can easily check that $\mathcal{C}$ is a compact mapping. Consider the mapping 
\[\mathcal{T}:\mathcal{B}(0, 1)\times\mathcal{C}(\mathcal{B}(0, 1))\to \mathcal{B}(0, 1),\]
defined by
\[\mathcal{T}(u(t), v(t))=a\cos(\omega t)+b\sin(\omega t)+\int_{0}^tG(t, s)f(s, u(s))v(s)ds.\]
Now,
\begin{align*}
D(t, u(s), v(s))e^{-\gamma t} & \leq \sup_{s\in I}~\Big\{\big|u_1(s)-\mathcal{T}(u_1(t), v(t))\big|+\big|u_2(s)-\mathcal{T}(u_2(t), v(t))\big|\Big\}e^{-\gamma t}\\
& =\Big\{\big\|u_1(t)-\mathcal{T}(u_1(t), v(t))\big\|+\big\|u_2(t)-\mathcal{T}(u_2(t), v(t))\big\|\Big\}e^{-\gamma (t-s)}
\tag{3.1}\label{eq3.1}
\end{align*}
and,
\begin{align*}
\big\|\mathcal{T}(u_1(t), v(t))-\mathcal{T}(u_2(t), v(t))\big\| & =\sup_{t\in I}\big|\mathcal{T}(u_1(t), v(t))-\mathcal{T}(u_2(t), v(t))\big|e^{-\gamma t}\\ & =\sup_{t\in I}\left|\int_{0}^tG(t, s)[f(s, u_1(s))-f(s, u_2(s))]v(s)ds\right|e^{-\gamma t}\\ & \leq \sup_{t\in I}\left|\int_{0}^tG(t, s)D(t, u(s), v(s))ds\right|e^{-\gamma t},
\end{align*} then from given condition $(i)$ and \eqref{eq3.1}, we have
\[\leq \Big\{\big\|u_1(t)-\mathcal{T}(u_1(t), v(t))\big\|+\big\|u_2(t)-\mathcal{T}(u_2(t), v(t))\big\|\Big\}\sup_{t\in I}\left|\int_{0}^tG(t, s)e^{-\gamma (t-s)}ds\right|\]


\begin{align*}
& =\Big\{\big\|u_1(t)-\mathcal{T}(u_1(t), v(t))\big\|+\big\|u_2(t)-\mathcal{T}(u_2(t), v(t))\big\|\Big\}\sup_{t\in I}\left|\int_{0}^t\frac{1}{\omega}\sin(\omega(t-s))e^{-\gamma (t-s)}ds\right|\\
& =\Big\{\big\|u_1(t)-\mathcal{T}(u_1(t), v(t))\big\|+\big\|u_2(t)-\mathcal{T}(u_2(t), v(t))\big\|\Big\}\sup_{t\in I}\left|\int_{0}^t\frac{1}{\omega}\sin(\omega r)e^{-\gamma r}dr\right|,\\
& (letting \quad r=t-s).
\tag{3.2}\label{eq3.2}
\end{align*}
Let 
\begin{align*}
L & =\int_{0}^t\frac{1}{\omega}\sin(\omega r)e^{-\gamma r}dr\\ & =\frac{1}{\omega}\int_{0}^t\sin(\omega r)e^{-\gamma r}dr\\ & =\frac{1}{\omega}\left[\frac{e^{-\gamma t}(\omega e^{\gamma t}-\gamma \sin(\omega t)-\omega \cos(\omega t))}{\gamma^2+\omega^2}\right].
\end{align*}
Now, $\sup_{t\in I}|L|=\dfrac{|\gamma +2\omega|}{|\omega|(\gamma^2+\omega^2)}.$ Let us assume that $k=\dfrac{|\gamma +2\omega|}{|\omega|(\gamma^2+\omega^2)}.$ So, $(iii)$ implies $k\in [0,\frac{1}{2})$.
On simplification, the inequality \eqref{eq3.2} produces \[\big\|\mathcal{T}(u_1(t), v(t))-\mathcal{T}(u_2(t), v(t))\big\|\leq k\Big\{\big\|u_1(t)-\mathcal{T}(u_1(t), v(t))\big\|+\big\|u_2(t)-\mathcal{T}(u_2(t), v(t))\big\|\Big\}.\] 
 Therefore, $\mathcal{T}$ satisfies Kannan type equicontraction condition. 
Let us take any $u(t)$ in $\mathcal{B}(0, 1),$ then by condition $(ii)$ \[\big\|\mathcal{T}(u(t), \mathcal{C}(u(t)))\big\|\leq |a|+|b|+\frac{M}{|\omega|^2}\leq 1.\]
It now follows from Theorem \ref{th2.4} that $\mathcal{T}(u(t), \mathcal{C}(u(t)))=u(t)$ has a solution in $\mathcal{B}(0, 1).$ Consequently, the above initial value problem has a solution in $\mathcal{B}(0, 1).$
\end{proof}

\subsection{Application II} Let us now consider another type of IVP:
\[\frac{d^2u}{dt^2}=u(t)+f(t, u(t))-\omega\left(\int_{0}^tG(t, s)\left[u(s)+f(s, u(s)\right]ds\right),\quad \omega\ne 0,\]
\[u(0)=0, u'(0)=0,\]
where each function, variable and constant satisfies the same properties as mentioned above. As an application of the theorem \ref{th2.8}, the existence of solution(s) of the immediate above problem is assured from the following theorem.
\begin{theorem}
Let $f:I\times\mathbb{R}\to \mathbb{R}$ be a nonzero function which is linear in second co-ordinate. There exists $M>0$ such that $\big|f(t,u)\big|\leq M,$ for all $t\in I$ and $u\in \mathbb{R}.$ The above initial value problem has a solution in closed unit ball $\mathcal{B}(0, 1)$ if the following conditions are satisfied
\begin{itemize}
\item[(i)]
$\big|f(s, u(s))-f(s, v(s))\big|\leq D(t, u(s), v(s))$ ~ where, 
\begin{align*}
D(t, u(s), v(s))= & \left|u_1(s)-v(s)-\int_{0}^tG(t, s)f(s, u_1(s))ds\right|+\\ & \left|u_2(s)-v(s)-\int_{0}^tG(t, s)f(s, u_2(s))ds\right|.
\end{align*}
\item[(ii)] $M\leq \frac{|\omega|}{8}-1$.
\item[(iii)] $|\gamma +2\omega|<\frac{|\omega(\gamma^2+\omega^2)|}{2}$.
\end{itemize}
\end{theorem}

\begin{proof}
If we take the same compact operator $\mathcal{C}:\mathbb{Y}\to \mathbb{Y}$ such that 
\[\mathcal{C}(u(t))=\int_{0}^tG(t, s)u(s)ds, \quad u\in \mathbb{Y}\]
where $G(t, s)$ has already been described.
Then $\mathcal{C}$ maps $\mathcal{B}(0, 1)$ to $\mathcal{B}\left(0, \dfrac{1}{|\omega|}\right).$ Consider the mapping $\mathcal{T}:\mathcal{B}(0, 1)\times\mathcal{B}(0,1/|\omega|)\to \mathbb{Y}$ such that
\[\mathcal{T}(u(t), v(t))=v(t)+\int_{0}^tG(t, s)f(s, u(s))ds,\quad where\quad v(t)\in\mathcal{C}(\mathcal{B}(0, 1)).\]
Take the family $\mathcal{S}_{v(t)}=\{v(t)\mapsto \mathcal{T}(u(t), v(t)):u(t)\in \mathcal{B}(0, 1)\}$ which is uniformly equicontinuous as if we consider $u(t)\in \mathcal{B}(0, 1)$ and $v_1(t), v_2(t)\in\mathcal{C}({\mathcal{B}(0, 1)}),$ corresponding to arbitrarily chosen $\varepsilon>0$ there exists $\delta>0$ such that $\|v_1(t)-v_2(t)\|<\delta.$ Now, 
\[\big\|\mathcal{T}(u(t), v_1(t))-\mathcal{T}(u(t), v_2(t))\big\|=\|v_1(t)-v_2(t)\|<\delta=\varepsilon,\]
Now, by similar calculation as previous theorem we can conclude that $\mathcal{T}$ satisfies Kannan type equicontraction condition. 
Consider $u(t)$ in $\mathcal{B}(0, 1),$ then by condition $(ii)$ 
\[\big\|\mathcal{T}(u(t), \mathcal{C}(u(t)))\big\|\leq \frac{1}{|\omega|}+\frac{M}{|\omega|}\leq \frac{1}{8}.\]
Here, The mapping $\eta:\mathcal{B}(0, 1)\to \mathcal{B}(0, 1)$ such that $\eta(u)=\mathcal{T}(u, \mathcal{C}(u))$ is clearly continuous and also it is $1st~invariant$ mapping as $\eta(u)$ is a linear map.
It now follows from Theorem \ref{th2.8} that $\mathcal{T}(u(t), \mathcal{C}(u(t)))=u(t)$ has a solution in $\mathcal{B}(0, 1).$ This assures the existence of solution(s) of the IVP in $\mathcal{B}(0, 1).$ 
\end{proof} 
In view of the obtained results, it is perhaps appropriate to end the present article with the following question:
\begin{Question}
By considering Kannan type equicontraction mapping and the idea of measure of noncompactness, can we conclude the proof of the Theorem \ref{th2.8} by eliminating the continuity condition of $\eta~\mathord{?}$
\end{Question}
\noindent \textbf{Conflicts of Interest} 

The authors declare that they have no conflicts of interest.
\vspace{0.1 cm}

\begin{Acknowledgement}
The first author is funded by University Grants Commission, Government of India. We are extremely grateful to Dr Saikat Roy and Mr Suprokash Hazra for their valuable suggestion to improve the Mathematical content as well as the writing of this project. 
\end{Acknowledgement}
\bibliographystyle{plain}

\begin{thebibliography}{10}
\bibitem{AKPRS}
R.R. Akhmerov, M.I. Kamenskii, A.S. Potapov, A.E. Rodkina, B.N. Sadovskii,
\newblock Measures of Noncompactness and Condensing Operators.
\newblock {\em Operator Theory: Advances and Applications}, vol. 55, 1992.

\bibitem{BG}
J. Bana\'{s}, K. Goebel,
\newblock Measures of noncompactness in Banach spaces.
\newblock {\em Lecture
Notes in Pure and Appl. Math. 60, Marcel Dekker, New York}, 1980.

\bibitem{BJMSV}
J. Bana\'{s}, M. Jleli, M. Mursaleen, B. Samet, C. Vetro, 
\newblock Advances in Nonlinear Analysis via the Concept of Measure of Noncompactness.
\newblock {\em Springer Nature Singapore Pte Ltd.}, 2017.

\bibitem{Bur}
T.A. Burton,
\newblock Integral equations, implicit functions, and fixed points.
\newblock{\em Proc. Amer. Math. Soc.}, 124(8), 2383-2390, 1996.

\bibitem{Dar}
G. Darbo,
\newblock Punti uniti in trasformazioni a codominio non compatto.
\newblock {\em Rend. Semin. Mat. Univ. Padova}, 24, 84-92, 1955.

\bibitem{Kara}
G.L. Karakostas, 
\newblock An extension of Krasnosel’ski\u{i}’s fixed point theorem for contractions and compact Mappings. 
\newblock{\em Topol. Methods Nonlinear Anal.}, 22, 181-191, 2003.

\bibitem{Kra}
M.A. Krasnosel’ski\u{i},
\newblock Some problems of nonlinear analysis, 
\newblock {\em American Mathematical Society Translations, Ser. 2}, 10, 345-409, 1958.

\bibitem{Pac}
M. P\u{a}curar,
\newblock Approximating common fixed points of Presi\'{c}-Kannan type operators by a multi-step iterative method.
\newblock{\em Analele Stiintifice ale Universitatii Ovidius Constanta, Seria Matematica}, 17(1), 153-168, 2009.

\bibitem{Pre}
S.B. Presi\'{c},
\newblock Sur une classe d’ in\'{e}quations aux diff\'{e}rences finite et sur la convergence de certaines suites.
\newblock{\em  Publ. Inst. Math. (Beograd)(N.S.)}, 5(19), 75-78, 1965.

\bibitem{Prz}
B. Przeradzki,
\newblock A generalization of Krasnosel’ski\u{i} fixed point theorem for sums of compact and contractible maps with application.
\newblock {\em Cent. Eur. J. Math.}, 10, 2012-2018, 2012.

\bibitem{Sad}
 B.N. Sadovski\u{i},
\newblock On a fixed point principle.
\newblock {\em Funktsional. Anal. i Prilozhen}, 1, 74-76, 1967.

\bibitem{Vet}
C. Vetro, D. Wardowski, 
\newblock Krasnosel’ski\u{i}-Schaefer type method in the existence problems. 
\newblock {\em Topol. Methods Nonlinear Anal.}, 54, 131-139, 2019.

\bibitem{War}
D. Wardowski,
\newblock Family of mappings with an equicontractive-type condition.
\newblock {\em J. Fixed Point Theory Appl.}, 22(55), 2020.

\bibitem{War1}
D. Wardowski,
\newblock A local fixed point theorem and its application to linear operators. 
\newblock {\em J. Nonlinear Convex Anal.}, 20, 2217–2223, 2019.

\bibitem{War2}
D. Wardowski, 
\newblock Solving existence problems via F-contractions. \newblock {\em Proc. Am. Math. Soc.}, 146, 1585-1598, 2018.


\end{thebibliography}

\end{document}